\documentclass[12pt]{article}
\usepackage{graphicx}
\usepackage{amsmath,amsxtra,amssymb,latexsym, amscd,amsthm}
\usepackage[unicode]{hyperref}
\usepackage{cleveref}
\usepackage{cite}
\usepackage{enumitem}
\theoremstyle{plain}
\newtheorem{theorem}{Theorem}[section]
\newtheorem{definition}{Definition}[section]
\newtheorem{example}{Example}[section]
\newtheorem{corollary}{Corollary}[section]
\newtheorem{lemma}{Lemma}[section]

\newtheorem{remark}{Remark}[section]

\begin{document}
\setlist{noitemsep}
\setlist{nolistsep}
\title{Some  Stability Properties of  Parametric Quadratically Constrained Nonconvex Quadratic Programs  in Hilbert Spaces}
\author{V. V. Dong\thanks{V. V. Dong, Phuc Yen College of Industry,  Vietnam,    vuvdong@gmail.com}     
} 
\date{accepted for publication in AMV, 09/6/2017}
\maketitle
\begin{quote}
\noindent {\bf Abstract.} {Stability of  nonconvex quadratic programming problems  under
	finitely many convex quadratic constraints in Hilbert spaces is investigated. We present  several stability properties of
	the global solution map, and the continuity of the optimal value function, 
	assuming that the problem data undergoes small perturbations. }

\end{quote}

\section{Introduction}
Let  $\cal H$ be a real Hilbert space with an inner product $\langle \cdot \,,  \cdot \rangle$
and its  induced norm denoted by $\Vert  \cdot  \Vert$. Let ${\cal L}(\cal H)$ be the  space of continuous linear operators  from $\cal H$ into $\cal H$ equipped with the operator norm induced by the vector norm in $\cal H$ and  also  denoted by $\Vert  \cdot  \Vert$.  
The norm in the product space $X_1 \times\ldots\times X_k$ of the normed spaces $X_1,\ldots, X_k$ is defined by $\Vert (x_1, \ldots,  x_k)\Vert = \max\{\Vert x_1 \Vert,\ldots, \Vert x_k\Vert\}.$ 
Let 
	$$\Omega :={\cal L(H)}\times {\cal H}\times {\cal L(H)}^{m}\times {\cal H}^{m}\times {\mathbb R}^m.$$

Consider the following   quadratic programming problem 
\begin{align}\label{eq:QP_w}
\begin{cases}
&\min \; f(x, \omega):=\frac{1}{2}\langle x, Tx\rangle  + \langle c, x\rangle\\
&\mbox{s. t.  }  x\in {\cal H}:   g_i(x,\omega):=\frac{1}{2}\langle x, T_i x\rangle + \langle c_i, x\rangle+\alpha_i\leqslant 0, \, i=1,\dots,m
\end{cases}\tag{{$QP_{\omega}$}}%
\end{align}
where $\omega =(T, c, T_1,\ldots,T_m, c_1,\ldots, c_m,\alpha_1,\ldots,\alpha_m)\in \Omega$, $T:{\cal H}\to {\cal H}$ is a continuous linear self-adjoint operator, $T_i$ is  a positive semidefinite continuous linear self-adjoint operator on $\cal H$,  $c, c_i \in \cal H$, and $\alpha$, $\alpha_i$ are real numbers, $ i=1,2,\dots,m$.

The \textit{constraint set} and the \textit{solution set} of \eqref{eq:QP_w} will be denoted by $F(\omega)$ 
and ${\rm Sol}\eqref{eq:QP_w}$, respectively.

Since $g_i$, $i=1,...,m$, are continuous and convex, $F(\omega)$  is closed and convex. Hence, by Theorem 2.23 in   \cite[ p. 24]{BS00}, the constraint set $F(\omega)$ of \eqref{eq:QP_w} is convex and weakly closed.

The  \textit{recession cone} of the constraint set of \eqref{eq:QP_w} can be described explicitly as follows (the proof is similar to the proof of Lemma 1.1 in \cite{KTY12}): If $F(\omega)$ is nonempty, then
	$${0^+}F(\omega)=\{v\in {\cal H}\; \mid\; T_iv=0, \langle c_i, v\rangle\leqslant 0,\
	\forall i=1,...,m\}.$$

The function
	$$\varphi: \Omega\to \mathbb R \cup \{+\infty\}$$
 defined by 
	$$  \varphi(\omega)= \begin{cases}
	\inf \{f(x, \omega) : x\in F(\omega)\}& \mbox{if } F(\omega) \neq \emptyset\\
	+\infty& \mbox{if } F(\omega) = \emptyset
	\end{cases}$$
is called the \textit{optimal value function} of the parametric problem \eqref{eq:QP_w}.

Quadratic programming problems (QP problems, in short) have been studied fairly completely in the setting of 
Euclidean spaces; see \cite{leetamyen05} and the references therein. For infinite dimensional spaces, it was extended to
Hilbert spaces. Existence of the solutions for QP problems in Hilbert spaces  have been investigated extensively in various versions; see 
\cite{Semple96,Schochetman97,BS00,Sivakumar03} and the references therein.

Stability is an important topic in optimization theory and practical applications. The continuity of the solution set mappings  and of the optimal value function in parametric optimization problems have been intensively studied in literatures; see, e.g., \cite{BS00,Berge} and the 
references therein. 
Bonnans and Shapiro \cite{BS00} gave sufficient conditions for
the upper semicontinuity of the solution set mapping and continuity of the optimal value function in parametric optimization problems  by assuming that the  level set is nonempty and contained in a compact set.   Berge \cite{Berge} gave a sufficient condition for  semicontinuity of the optimal value function in parametric optimization problems.

Since QP problems  form a subclass of nonlinear 
optimization problems, the stability results in nonlinear optimization can be applied to  QP problems in Hilbert spaces.  However, the special structure of QP problems allows one to
have deeper and  sharper results on stability properties of QP problems.

This paper studies  parametric quadratic programming problems in a Hilbert space. The main results of the paper concern continuity properties of the solution map and the optimal value function of the problem whose quadratic part of the objective function is a Legendre form and the constraints are convex under Slater's condition.   
Our results can be seen as an extension of those  in \cite{leetamyen05,Tam99,Tam01} and the references therein to Hilbert spaces.

The paper is organized as follows. In Section 2 we study continuity of  the solution map in a parametric QP problem.
Continuity properties  of the optimal value function of the problem \eqref{eq:QP_w} under a perturbation are investigated in  Section 3.

\section{Continuity of the Global Solution Map }

In this section, we are going to study  continuity properties  of the solution set mapping ${\rm Sol}( \cdot ): \Omega\rightrightarrows {\cal H}$ of \eqref{eq:QP_w}   defined by
	$$  {\rm Sol}(\omega)=\{x\in F(\omega)\mid f(x, \omega) = \varphi(\omega) \}.$$

\begin{definition}
	Let $S : X \rightrightarrows Y$ be a set-valued map from Hilbert space $X$ to Hilbert space $Y$. It is said that $S$ {\it is upper semicontinuous} (usc) at $\bar u\in X$  if for each open set $V \subset Y$ satisfying $S(\bar u)\subset V$, there exists $\varepsilon> 0$  such that $S(u)\subset V$ whenever $\Vert u-\bar u\Vert<\varepsilon$. If for each open set $V\subset Y$ satisfying 	$S(\bar u) \cap  V\neq \emptyset$ there exists $\varepsilon> 0$ such that $S(u) \cap V\neq \emptyset$ whenever $\Vert u-\bar u\Vert<\varepsilon$, then $S$ is 	said to be {\it lower semicontinuous} (lsc) at $\bar u\in X$. If $S$ is simultaneously usc and lsc at $\bar u$, we	say that it is {\it continuous} at $\bar u$.
\end{definition}

The inequality system \; 
$  g_i(x, \omega)\leqslant 0,\; i=1,\ldots, m,$
is called regular if there exists  $x^0\in {\cal H}$ such that $g_i(x^0, \omega)<0, \; i=1,\ldots, m$.

\medskip

In this paper, we will only consider  the continuous quadratic forms in the  form  $Q(x) = \langle x, T x\rangle,$
where $T: {\cal H}\to {\cal H}$ is a continuous linear self-adjoint operator. 

\begin{definition}\label{de:Legendre}\rm {(see \cite[p. 551]{Hestenes1})}
	A quadratic form $Q:{\cal H}\to \mathbb R$ is said to be a Legendre form if  it is weakly lower semicontinuous and $x_k\to x_0$ whenever $x_k$ weakly converges to $x_0$ and $Q(x_k)\to Q(x_0)$.
\end{definition}

It is clear that in the case where ${\cal H}$ is of finite dimension, any quadratic form $Q(x)$ on ${\cal H}$ is a Legendre form.
It is easy to see that on infinite dimensional Hilbert spaces, the quadratic form $\langle x,Ix\rangle$  is a Legendre form while the identity operator  $I$ is noncompact, the quadratic form  $\langle x,0x\rangle$ is not a Legendre form while  the zero operator $0$ is  compact.

\medskip

For each  problem \eqref{eq:QP_w},  we consider the following problem
\begin{align}\label{eq:P_0}
\min \{ \frac{1}{2} \langle v, Tv \rangle : v\in 0^+F(\omega)\}. \tag{${QPR_{\omega}}$}
\end{align}
Let us denote by  
${\rm Sol}\eqref{eq:P_0}$ the solution set of  \eqref{eq:P_0}.
The problem \eqref{eq:P_0} is closely related to \eqref{eq:QP_w}. The solution set of  \eqref{eq:P_0}
plays an important role in the study of the stability of the problem \eqref{eq:QP_w}.

\medskip

The following lemmas will be useful in the sequel.

\begin{lemma}\label{remark1}
	Suppose that 
	$x^k\in F(\omega)\backslash\{0\}$ for all $k$, $\Vert x^k\Vert\to\infty$ as $k\to\infty$ and $\Vert x^k\Vert^{-1}x^k$ weakly converges to $\bar v$. Then, $\bar v\in 0^+F(\omega).$
\end{lemma}
\begin{proof}
	Since $x^k\in F(\omega)$, we have
	\begin{align}\label{eqlema:rec2_1}
	\frac{1}{2}\langle x^k,T_ix^k\rangle +\langle c_i,x^k\rangle+\alpha_i \leqslant 0\;\; i=1,...,m.
	\end{align}
	Multiplying both sides of  the inequalities in  \eqref{eqlema:rec2_1} by $\Vert x^k\Vert^{-2}$ and  letting $k\rightarrow \infty$, we obtain
	\begin{equation*}
	\underset{k\to\infty}{\liminf}\frac{1}{2}\left\langle \frac{x^k}{\Vert x^k\Vert}, T_i\frac{x^k}{\Vert x^k\Vert}\right\rangle \leqslant 0,\; i=1,...,m.
	\end{equation*}
	Since $T_i$ is positive semidefinite, by Proposition 3 in  \cite[p.~269]{Ioffe}, $\langle x,T_ix\rangle$ is  weakly lower semicontinuous.  Hence, 
	\begin{equation*}
	\frac{1}{2}\langle \bar v, T_i\bar v\rangle \leqslant \underset{k\to\infty}{\liminf}
	\frac{1}{2}\left\langle \frac{x^k}{\Vert x^k\Vert}, T_i\frac{x^k}{\Vert x^k\Vert}\right\rangle \leqslant 0\;\; i=1,...,m.
	\end{equation*}
	By the positive semideniteness of  $T_i$, from this we can deduce that
	\begin{equation} \label{eqlema:rec2_2}
	T_i\bar v=0\quad\, \forall i=1,...,m.
	\end{equation}
	As $\langle x^k,T_ix^k\rangle \geqslant 0$, from \eqref{eqlema:rec2_1} it follows
	that
	\begin{equation*}
	\langle c_i, x^k \rangle+\alpha_i \leqslant 0,\quad \forall i=1,..., m,\ \, \forall k.
	\end{equation*}
	Multiplying the inequality $\langle c_i, x^k \rangle+\alpha_i \leqslant 0$ by
	$\Vert x^k\Vert^{-1}$ and letting $k\rightarrow \infty$, we get
	\begin{equation}\label{eqlema:rec2_3}
	\langle c_i,\bar v \rangle \leqslant 0\quad\, \forall i=1,...,m.
	\end{equation}
	Combining  \eqref{eqlema:rec2_2} with
	\eqref{eqlema:rec2_3} we obtain $\bar v\in 0^+F(\omega)$. \qed
\end{proof}

\begin{lemma}\label{theo:sol_nonemp}
	Consider the problem  \eqref{eq:QP_w}, where  $\langle x, Tx\rangle$  is a Legendre form. Suppose that  $F(\omega)$ is nonempty and ${\rm Sol}\eqref{eq:P_0} = \{0\}$.  Then, ${\rm Sol}(\omega)$  is a nonempty,  closed and bounded set.
\end{lemma}
\begin{proof}
	We first  prove that ${\rm Sol}(\omega)$  is a nonempty set. By \cite[Theorem 2]{DongTam15a}, it suffices to show that $f(x, \omega)$ is bounded from below over $F(\omega)$.  
	On the contrary, suppose that $f(x, \omega)$ is unbounded from below over $F(\omega)$.  Then,  there exists a sequence $\{x^k\}\subset F(\omega)$ 
	such that $f(x^k)\to -\infty$ as $k\to\infty$. If $\{x^k\}$ is bounded,  it has a
	weakly convergent subsequence. Without loss of generality, we can assume
	that $x^k $ itself weakly converges to some $\bar x$. By the weakly closedness of $F(\omega)$, we have ${\bar x}\in F(\omega)$. Since $\langle x, T x\rangle$ is a Legendre form, it is weakly lower semicontinuous, one has
	$f(\bar x)\leqslant \underset{k\to\infty}{\liminf}f(x^k)=-\infty$, which is impossible. Hence $\{x^k\}$ is unbounded.  
	Since  $f(x^k)\to -\infty$ as $k\to\infty$ and $\{x^k\}$ is unbounded, 
	by taking a subsequence, if necessary, we may assume that 
	\begin{align}\label{eq:theo_sol_nonemp1}
	f(x^k,\omega)= \frac{1}{2}\langle x^k, T x^k\rangle + \langle c, x^k\rangle \leqslant 0
	\end{align}
	for all $k$, $\Vert x^k\Vert\neq 0$, $\Vert x^k\Vert\to \infty$ as $k\to\infty$, 
	and $v^k:=\Vert x^k\Vert^{-1} x^k$ weakly converges to some $\bar v$ as $k\to\infty$.
	It follows from Lemma \ref{remark1} that
	$\overline{v}\in 0^+F(\omega)$. 
	Multiplying both sides of \eqref{eq:theo_sol_nonemp1} by $\Vert x^k\Vert^{-2}$ and letting $k\to\infty$,
	one has 
			\begin{align}\label{eq:theo_sol_nonemp2}
	\langle \bar v, T \bar v\rangle\leqslant 
	\underset{k\to\infty}{\liminf} \frac{1}{2}\langle v^k, T v^k\rangle\leqslant 
	\underset{k\to\infty}{\limsup} \frac{1}{2}\langle v^k, T v^k\rangle \leqslant 0.
	\end{align}
		We next claim  that  $\bar v \neq 0$. Indeed, if  $\bar v = 0$, then it follows from \eqref{eq:theo_sol_nonemp2} that $ \langle v^k,  T v^k\rangle \to  \langle \bar v,T \bar v\rangle$ as $k\to \infty$.  Since $\langle x, T x\rangle$ is a Legendre form, and since  $v^k\rightharpoonup \bar v$, $ \langle v^k,  T v^k\rangle \to  \langle \bar v,T \bar v\rangle$ as $k\to \infty$, by Definition \ref{de:Legendre} we deduce that  $v^k$  converges to $\bar v$ and $\bar v\neq 0$.  Consequently, we have shown that there exists $\bar v \neq 0$ such that  $\bar v\in 0^+F(\omega)$ and $\langle \bar v, T \bar v\rangle  \leqslant 0$.
	But this contradicts our assumption that ${\rm Sol}\eqref{eq:P_0}=\{0\}$. Hence $f(x, \omega)$ is bounded from below over $F(\omega)$ and we have ${\rm Sol}(\omega)$  is a nonempty set.	

	The closedness of  ${\rm Sol}(\omega)$ is evident because $f(x, \omega)$ is
	 weakly lower semicontinuous  and $F(\omega)$ is a  closed convex set.
	
	We next prove that ${\rm Sol}(\omega)$ is bounded.
	 Suppose that ${\rm Sol}(\omega)$ is unbounded.
	Then,  there exists  $\{y^k\}\subset {\rm Sol}(\omega)$ such that  $\Vert y^k\Vert \to \infty$ as $k\to \infty$. Without loss generality we may suppose that $\Vert y^k\Vert\neq 0$.
	Let $v^k:=\frac{y^k}{\Vert y^k\Vert}$, one has $\Vert v^k\Vert = 1$. Since ${\cal H}$  is Hilbert space, extracting if necessary a subsequence, we may assume that $v^k$ itself weakly converges to some $\overline{v}$. It follows from Lemma \ref{remark1} that
$\overline{v}\in 0^+F(\omega)$.
	
	Fixing any $x\in F(\omega)$,  one has
	\begin{align}\label{eq:prbound_1}
	\frac{1}{2}\langle y^k, T y^k \rangle  +\langle c, y^k\rangle\leqslant \frac{1}{2}\langle x, T x\rangle + \langle c, x\rangle.
	\end{align}

	Since  $\langle x, T x\rangle$ is a Legendre form, it is weakly lower semicontinuous.
Dividing both sides of \eqref{eq:prbound_1} by $\Vert y^k\Vert^2$ and letting $k\to\infty$,  we get
\begin{align}\label{eq:prbound_2}
\langle \bar v, T\bar {v}\rangle\leqslant \underset{k\to\infty}{\lim\inf}\,\langle v^k,  T v^k\rangle
\leqslant\underset{k\to \infty}{\lim\sup}\,\langle v^k, T v^k \rangle  \leqslant 0.
\end{align}
By a similar argument to the one given above, we have 
$\bar v\in 0^+F(\omega)\backslash\{0\}$ and $\langle \bar v, T \bar v\rangle  \leqslant 0$, contrary to 
${\rm Sol}\eqref{eq:P_0}=\{0\}$. Hence 
${\rm Sol}(\omega)$ is bounded. \qed
\end{proof}

Note that,  the problem \eqref{eq:QP_w}  may have no solution if the assumption on the Legendre property of the quadratic
form is omitted (see  \cite[Example 3.3]{DongTam15b}).

The next example shows that the conclusion of Lemma \ref{theo:sol_nonemp} fails if the assumption on the Legendre property of the quadratic form is omitted.
\begin{example}\label{exam:unbounded}\rm
	Consider the  programming problem
	\begin{align}\label{ex:unboundedSet}
	\begin{cases}
	& \min  f(x, \omega)= \frac{1}{2}\langle x, T x\rangle \\
	&\mbox{ s. t. }   x\in L_2[0,1],  g_1(x, \omega) = \frac{1}{2} \langle x, T_1 x\rangle + \alpha_1 \leqslant  0,
	\end{cases}
	\end{align}
	where $T= 0$,  $T_1: L_2[0,1]\to L_2[0,1]$ is defined by $T_1x = tx(t)$, $\alpha_1 = -  \frac{1}{4} $ and $\omega = (T, T_1, \alpha_1)$. 
	
	It is easily seen that 
	$\langle x, T x\rangle = \langle x, 0 x\rangle$ is not a Legendre form.
		
	Let 
	$F(\omega)=\{ x\in L_2[0,1] \mid g_1(x, \omega) = \frac{1}{2} \int\limits_0^1 t x^2(t)dt -  \frac{1}{4}\leqslant  0\}.$
	 
	It is  easy to check that $F(\omega)\neq \emptyset$  and 
	$$0^+F(\omega) = \{v\in L_2[0, 1]\mid tv(t) = 0\; \forall t\in [0, 1]\} = \{0\}.$$
	Therefore ${\rm Sol}\eqref{eq:P_0} = \{0\}$.
	
	Let 
		$$x_k(t) = \begin{cases}
		k &  \mbox{ if\;\; } 0\leqslant t\leqslant \frac{1}{k} \\
		0 & \mbox{ if\;\; }  \frac{1}{k} < t \leqslant 1,
	\end{cases}$$
		it is easy to check that $x_k(t)\in L_2[0,1]$. 
	
	We have $g_1(x_k, \omega) = \frac{1}{2} \int\limits_0^{\frac{1}{k}} t k^2 dt -  \frac{1}{4}=  0$ and $\Vert x_k\Vert^2 = \int\limits_0^{\frac{1}{k}}  k^2 dt =k$. Hence $x_k\in F(\omega)$ and $\Vert x_k\Vert\to \infty$ as $k\to\infty$.
	
	Since $f(x,\omega) =0$ for all $x\in F(\omega)$, it follows that  the solution set of \eqref{ex:unboundedSet} coincides with $F(\omega)$. Thus the  solution set of \eqref{ex:unboundedSet} is unbounded.
\end{example}
\begin{corollary}
	Consider the problem  \eqref{eq:QP_w}, where  $\langle x, T x\rangle$  is a nonnegative  Legendre form. Assume that  $F(\omega)$ is nonempty and ${\rm Sol}\eqref{eq:P_0}=\{0\}$.
	Then,   ${\rm Sol}(\omega)$ is nonempty and weakly compact. 
\end{corollary}
\begin{proof}
	It follows from Lemma \ref{theo:sol_nonemp}  that ${\rm Sol}(\omega)$ is a nonempty and bounded set. Since $\langle x, T x\rangle$  is  nonnegative, it follows that  \eqref{eq:QP_w} is a convex problem. Hence ${\rm Sol}(\omega)$ is a convex set. By Theorem 3.3 in 
	\cite{Bauschke2011},  ${\rm Sol}(\omega)$ is weakly compact. \qed
\end{proof}

\begin{lemma}{\rm {\cite[Theorem 5.1]{Dinh}}}\label{pro-lscF}
	Let $\omega\in\Omega$. 
	If the system $g_i(x,\omega)\leqslant 0, i=1,dots, m$, is regular, then the set-value map 
	$F\, (\cdot): \Omega \rightrightarrows {\cal H}$ is defined by 
	$F\, (\omega^{\prime}) = \{x\in {\cal H}\mid g_i(x, \omega^{\prime})\leqslant 0,\; i=1,\ldots, m\}$ is lower semicontinuous at $\omega$.
\end{lemma}

\begin{remark}\label{remark:lscF}
	If the inequality system $g_i(x, \omega)\leqslant 0, i=1,\ldots, m$, is irregular, then there exists a sequence $\{\omega^k\}$ in $\Omega$ converging to $\omega$ such that, for every $k$, the system $g_i(x, \omega^k)\leqslant 0,\; i=1,\ldots, m$, has no solution. 
\end{remark}

\begin{lemma}\label{le:converge}
	Consider  the parametric quadratic function on Hilbert space $\cal H$
		$$h(x, u):=\frac{1}{2}\langle x, T x\rangle + \langle c, x\rangle + \alpha,$$
	where $\langle x, T x\rangle$ is weakly lower semicontinuous, $c\in {\cal H}$, $\alpha\in \mathbb R$ and\\ $u= (T, c, \alpha)\in \Omega_1:= {\cal L}({\cal H})\times {\cal H}\times {\mathbb R}$. 
	
	Suppose that  there exists a sequence
	$\{u^k\}=\{(T^k, c^k, \alpha^k)\}$ in $\Omega_1$  converging to $u$ 
	and $x^k$ weakly converges to $\bar x$ as $k\to\infty$. Then, 
	\begin{itemize}
		\item [{\rm (a)}]
		$\underset{k\to\infty}{\liminf}\,h(x^k, u^k)\geqslant h(\bar x, u)$;
		\item [{\rm (b)}] If $\langle x^k, T^k x^k\rangle\to\langle \bar x,  T \bar x\rangle$ as $k\to\infty$ then $\langle x^k,  T x^k\rangle\to\langle \bar x,  T \bar x\rangle$  as $k\to\infty$;
		\item [{\rm (c)}] If $T^k x^k = 0,\; \langle c^k, x^k\rangle \leqslant 0$ and $\langle x,  T x\rangle$ is nonnegative,  then $$T \bar x = 0, \langle \bar c, \bar x\rangle \leqslant 0.$$
	\end{itemize}
\end{lemma}
\begin{proof}
	${\rm (a)}$ 
	Let us first prove  that $\underset{k\to\infty}{\liminf}\langle x^k, T^k x^k\rangle\geqslant \langle \bar x, T\bar x\rangle$. We have 
	\begin{align}\label{eq:lema_1a}
	\langle x^k, T^k  x^k \rangle - \langle \bar x, T \bar x\rangle 
	= \langle x^k, (T^k -T) x^k\rangle+ [\langle x^k,   T x^k\rangle -\langle \bar x,  T \bar x\rangle].
	\end{align}
	Since  $\Vert T^k -  T\Vert \to 0$ as $k\to \infty$ and since $\Vert x^k\Vert$ is bounded, by  the Cauchy-Schwarz inequality (see \cite[p. 29]{Bauschke2011})  we see that 
	\begin{align}\label{eq:lema_1b}
	\vert\langle x^k, (T^k - T) x^k\rangle\vert \leqslant \Vert x^k\Vert^2\Vert T^k- T\Vert \to 0 \mbox{ as } k\to\infty.
	\end{align}
	By \eqref{eq:lema_1a}, \eqref{eq:lema_1b} and the weakly lower semicontinuity of $\langle x, T x\rangle$, one has
	\begin{align}\label{eq:lema_1c}
	\underset{k\to\infty}{\liminf}\langle x^k, T^k x^k\rangle\geqslant \langle \bar x,  T\bar x\rangle.
	\end{align}
	Combining \eqref{eq:lema_1c}  with $\underset{k\to\infty}{\lim}(\langle c^k, x\rangle + \alpha^k)= \langle c, x\rangle +  \alpha \mbox{ as } k\to\infty$ we can assert that $\underset{k\to\infty}{\liminf}\,h(u^k, x^k)\geqslant h( u, \bar x).$
	
	${\rm (b)}$ Since $\langle x^k,  T x^k\rangle-\langle \bar x,  T \bar x\rangle = \langle x^k, ( T-T^k) x^k\rangle + [\langle x^k, T^k x^k\rangle-\langle \bar x,   T\bar x],$
	it follows that
	\begin{align}\label{eq:lema_1d}
	\Vert\langle  x^k, T x^k\rangle-\langle \bar x,   T\bar x\rangle\Vert \leqslant \Vert\langle x^k, (T-T^k) x^k\rangle\Vert+\Vert  \langle x^k, T^k x^k\rangle-\langle \bar x,  T\bar x)\Vert.
	\end{align}
	From \eqref{eq:lema_1b}, \eqref{eq:lema_1d} and the assumption $\langle x^k, T^k x^k\rangle\to\langle \bar x, T \bar x\rangle$ as $k\to \infty$, it may be concluded that  $\langle x^k,  T x^k\rangle\to\langle x,  T x\rangle$  as $k\to\infty$.
	
	${\rm (c)}$ Obviously, $\langle  c, \bar x\rangle = \underset{k\to\infty}{\lim}\langle c^k, x^k\rangle \leqslant 0$.  By the assumption $T^k x^k =0$ and  $\langle x,  T x\rangle$ is nonnegative, 	it follows from \eqref{eq:lema_1c},  that  
	$ \langle \bar x, T \bar x\rangle = 0$.  By the well-known Fermat Rule, we obtain that   $ T \bar x=0$.  The proof is complete. \qed
	 \end{proof}
\begin{lemma}\label{pro:open}
	Consider the problem \eqref{eq:QP_w}, where 
	$\langle x, Tx\rangle$ is a Legendre form.
	Then, set
		$$K:=\big\{(T, T_1,\ldots, T_m , c_1, \ldots, c_m)\in {\cal L(H)}^{m+1}\times {{\cal H}^m}\mid {\rm Sol}\eqref{eq:P_0} =\{0\}\big\}$$
	is open  in  ${\cal L(H)}^{m+1}\times {{\cal H}^m}$.
\end{lemma}
\begin{proof}
	
	On the contrary, suppose that $K$ is not open. Then, there exists a sequence
		$$\{(T^k, T^k_1,\ldots,T^k_m, c^k_1,\ldots,c^k_m)\}\subset {\cal L(H)}^{m+1}\times {{\cal H}^m}$$
	converging to $(T, T_1,\ldots, T_m, c_1,\ldots, c_m)\in K$  and a sequence   $v^k\in {\cal H}, \Vert v^k\Vert \neq 0$ such that
	\begin{align}\label{eq:open2}
	T_i^k v^k = 0, \;\;\langle c_i^k, v^k\rangle \leqslant 0, \;\;i=1,\ldots, m,
	\; \langle v^k,  T^k v^k\rangle\leqslant 0.
	\end{align}
	Let $h^k :=\frac{v^k}{\Vert v^k\Vert}$.
	By taking a subsequence, if necessary, we can assume that $h^k$ weakly converges to $\overline{v}$. It follows  from Lemma \ref{le:converge} that $ \overline{v}\in 0^+F(\omega)$. 
	
	Dividing both sides of  the inequalities 
	$\langle v^k, T^k v^k\rangle \leqslant 0$ in \eqref{eq:open2} by  $\Vert v^k\Vert^2$ and letting  $k\to\infty$, one has
	$$\underset{k\to\infty}{\limsup}\langle v^k,  T^k v^k \rangle \leqslant 0.$$
	Combining this with Lemma \ref{le:converge} we deduce that
	\begin{align}\label{eq:open_2a}
	&\langle \bar{v}, T \bar{v}\rangle\leqslant \underset{k\to\infty}{\liminf}\langle v^k,  T^k v^k \rangle\leqslant \underset{k\to\infty}{\limsup}\langle v^k,  T^k v^k \rangle = 0.
	\end{align}
	We next claim  that  $\bar v \neq 0$. Indeed, if $\bar v = 0$, then it follows from \eqref{eq:open_2a} and Lemma \ref{le:converge} that $ \langle v^k,  T v^k\rangle \to  \langle \bar v,T \bar v\rangle$ as $k\to \infty$.  Since $\langle x, T x\rangle$ is a Legendre form, and since  $v^k\rightharpoonup \bar v$, $ \langle v^k,  T v^k\rangle \to  \langle \bar v,T \bar v\rangle$ as $k\to \infty$, by Definition \ref{de:Legendre} we deduce that  $v^k$  converges to $\bar v$ and $\bar v\neq 0$,  a contradiction.
	
	We have thus proved that there exists
	$\overline{v}\in 0^+F(\omega)\backslash \{0\}$ such that
	$\langle \overline{v}, T \overline{v}\rangle \leqslant 0$. This contradicts the assumption that ${\rm Sol}\eqref{eq:P_0} = \{0\}$. Hence $K$ is open. The proof is complete. \qed
\end{proof}
The next example shows that the assumption on the Legendre property of the quadratic form cannot be dropped from the assumption of Lemma \ref{pro:open}.
\begin{example}\label{exam:not_open}\rm Consider the programming problem
	\begin{align}\label{eq:not_open}
	\begin{cases}
	&\min  f(x, \omega)= \frac{1}{2}\langle x, T x\rangle +\langle c, x\rangle \\
	&\mbox{ s. t. } x=(x_1, x_2,\ldots)\in \ell^2, \, g_1(x, \omega) = \frac{1}{2}\langle x, T_1 x\rangle +\langle c_1, x\rangle+\alpha_1\leqslant  0,
	\end{cases}
	\end{align}	
	where $T= 0$, $c= 0$, $T_1: \ell^2\to\ell^2$ is defined by $T_1x = (x_1, \frac{x_2}{2^2}, \ldots, \frac{x_n}{n^n},\ldots)$, $c_1 = 0$, $\alpha_1 = -1$ and $\omega = (T, c, T_1, c_1, \alpha_1)$. 
	
	Let 
		$$F(\omega) = \{x\in \ell^2\mid g_1(x, \omega) = \frac{1}{2}\langle x, T_1 x\rangle+ \langle c_1, x\rangle+\alpha_1\leqslant  0\}$$
	and let
		$$K= \{(T, T_1, c_1)\in {\cal L}(\ell^2)^2\times {\ell^2} \mid {\rm Sol}(\eqref{eq:P_0}) = \{0\}\}.$$
		
	It is easily seen that $F(\omega)$ is a nonempty set and $\langle x, T x\rangle= \langle x, 0 x\rangle$
	is not a Legendre form. 
	
	The quadratic form associated with $T_1$ given by $\langle x, T_1 x\rangle = \sum\limits_{n=1}^{\infty}\frac{x_n^2}{n^n}$.
		Since $\langle x, T_1 x\rangle =\sum\limits_{n=1}^{\infty}\frac{x_n^2}{n^n} = 0$ if and only if $x=(0,\ldots, 0,\ldots)$, it follows that 
		$$0^+F(\omega) = \{v\in \ell^2\mid T_1 v = 0\}=\{0\}.$$
		 Hence ${\rm Sol}\eqref{eq:P_0} =\{0\}$.
	
	Let $T_1^n = T_1- \frac{1}{n^n}I$, where $I$ is the identity operator on $\ell^2$. 
	
	Since $\Vert T_1- T_1^n\Vert = \frac{1}{n^n}$, we have $\Vert (T, T_1^n) - (T, T_1)\Vert = \Vert T_1^n-T_1\Vert \to 0$ as $n\to\infty$. 
		It is clear that $T^n\in \ell^2$ and
		$$T_1^n x=\Big(x_1+\ldots (\frac{1}{n^{n-1}}- \frac{1}{n^n})x_{n-1}+ 0+ (\frac{1}{n^{n+1}}- \frac{1}{n^n} )x_{n+1}+\ldots\Big) \mbox{ \, for all } n.$$ 
		Note that if $\, \overline{v}^n = (0, 0,\ldots, v_n, 0,\ldots), v_n\neq 0$, then $ T_1^n \overline{v}^n = 0$. Therefore ${\rm Sol}(QPR_{\omega^n})\neq \{0\}$. 
		
	We have shown that there exists a sequence $\{(T, T_1^n, c_1)\}\subset {\cal L}(\ell^2)^2\times \ell^2$ converging to $(T,  T_1, c_1)\in K$ such that ${\rm Sol}(QPR_{\omega^n})\neq \{0\}$. Hence $K$ is not open.
\end{example}

A sufficient condition for the upper semicontinuity
of $\rm Sol( \cdot )$ is as follows.

\begin{theorem}\label{theo:usc}
	Consider the problem \eqref{eq:QP_w} where $\langle x, T x\rangle$ is a Legendre form. Then,  the multifunction ${\rm Sol}(\cdot)$ is usc at $\omega$ if  the following two conditions are satisfied:
	\begin{itemize}
		\item [{\rm (i)}]  ${\rm Sol}\eqref{eq:P_0} =\{0\}$;
		\item [{\rm (ii)}] The system $g_i(x, \omega)\leqslant 0, i=1, \dots, m,$ is regular.
	\end{itemize}
\end{theorem}
\begin{proof}
	Suppose that the assertion of the theorem is false.
	Then,  there exist an open set $V$ containing ${\rm Sol}(\omega)$, a sequence $\{\omega^k\}$ converging to $\omega$, a sequence $\{x^k\}$ such that
	$x^k\in {\rm Sol}(\omega^k)\backslash V\; $ for all $k$.
	Since $x^k\in {\rm Sol}(\omega^k)$, one has
	\begin{align}\label{eq:nlt_tren1}
	g_i(x^k, \omega^k)\leqslant 0,\; i=1,\ldots,m.
	\end{align}
	
	Fix any $\bar x\in F(\omega)$. By assumption ${\rm (ii)}$ and Lemma \ref{pro-lscF}, there exists a sequence 
	$\{\xi^k\},\;\; \xi^k\in F(\omega^k)$ for all $k$
	such that $\underset{k\to\infty}{\lim}\xi^k= \bar x$.
	We have
	\begin{align}\label{eq:nlt_tren2}
	\frac{1}{2}\langle x^k, T^k x^k \rangle+\langle c^k, x^k\rangle\leqslant \frac{1}{2}\langle \xi^k, T^k\xi^k\rangle+\langle c^k, \xi^k\rangle .
	\end{align}
	
	If the sequence $\{x^k\}$ is bounded, then there is no loss of generality
	in assuming that
	$x^k\rightharpoonup x_0\in {\cal H}$.
	By  passing  to  the limit in \eqref{eq:nlt_tren1} and \eqref{eq:nlt_tren2} as $k\to\infty$,   and using  Lemma \ref{le:converge},  we obtain $g_i(x_0, \omega)\leqslant 0$ and $f(x_0, \omega)\leqslant f(\overline{x}, \omega)$. Hence $x_0\in {\rm Sol}(\omega)\subset V$. We have arrived at a
	contradiction, because $x_0\notin V$ for all $k$, and $V$ is open.
	
	Suppose now that $\{x^k\}$ is unbounded. Without loss of
	generality we can assume that $\Vert x^k\Vert\neq 0$ for all $k$ and  $\Vert x^k\Vert\to\infty$ as $k\to\infty$.
	Let $v^k:=\frac{x^k}{\Vert x^k}\Vert$, we have $\Vert v^k\Vert = 1$, by taking a subsequence, if necessary, we
	can assume that $v^k\rightharpoonup v$ as $k\to\infty$.
	It follows from Lemma \ref{remark1} that
	$v\in 0^+F(\omega)$.
	
	Dividing both sides of  \eqref{eq:nlt_tren2} by $\Vert x^k\Vert^2$ and letting $k\to\infty$, by Lemma \ref{le:converge}${\rm (a)}$  we can deduce that
		$$\langle v,  T v\rangle \leqslant \underset{k\to\infty}{\lim\inf}\langle v^k, T^k v^k\rangle\leqslant 
		\underset{k\to\infty}{\lim\sup}\langle v^k, T^k v^k \rangle\leqslant 0.$$
	In the same way as in the proof of Lemma \ref{pro:open}, we obtain  $v\in 0^+F(\omega)\backslash \{0\}$ and $\langle v, Tv \rangle \leqslant 0$, a contradiction.
	The proof  is complete. \qed
\end{proof}
\begin{remark}
	Observe that neither $a)$ nor $b)$ is a necessary condition for the upper semicontinuity of the ${\rm Sol}( . )$ at a given $\omega$ (see, \cite[Example 12.1 and 12.2 ]{leetamyen05}).
\end{remark}
\begin{corollary}
	Consider the problem \eqref{eq:QP_w}, where   $\langle x, T x\rangle$   is a Legendre form. Suppose that  the system $g_i(x,\omega)\leqslant 0,\; i=1,\ldots, m$, is regular and $F(\omega)$ is bounded. Then, 
	${\rm Sol}(\cdot)$ is usc at $\omega$.
\end{corollary}
\begin{proof}
	Since the system $g_i(x, \omega)\leqslant 0,\; i=1,\ldots, m$, is regular, $F(\omega)\neq\emptyset$. The boundedness of $F(\omega)$ implies that $0^+F(\omega)=\{0\}$. 
	Hence ${\rm Sol}\eqref{eq:P_0} =\{0\}$, and the desired property follows from 
	Theorem \ref{theo:usc}. The proof is complete. \qed
\end{proof}

The following example shows that  the assumption on the Legendre property of the quadratic form cannot be dropped from the assumption of Theorem \ref{theo:usc}.

\begin{example}\label{exam:not_usc}\rm 
	Consider the problem
	\begin{align}\label{eq:not_usc}
	\begin{cases}
	&\min  f(x, \omega)= \frac{1}{2}\langle x, T x\rangle \\
	&\mbox{ s. t. } x=(x_1, x_2,\ldots)\in \ell^2,  \, g_1(x, \omega) =  \langle c_1, x\rangle + \alpha_1 \leqslant  0,
	\end{cases}
	\end{align}	
	where $T: \ell^2\to\ell^2$ is defined by $$Tx = (x_1, \frac{x_2}{2^2}, \ldots, \frac{x_n}{n^n},\ldots),$$ $c_1=(-1, -\frac{1}{2},\ldots, -\frac{1}{n},\ldots)$, $\alpha_1 = 1$ and $\omega = (T,  c_1, \alpha_1)$.
	
	Let $F(\omega) = \{x\in\ell^2\mid g_1(x, \omega) =  \langle c_1, x\rangle + \alpha_1 \leqslant  0\}$ and let ${\rm Sol}(\omega)$ denote the solution of \eqref{eq:not_usc}.
	
	Since $x=(1,0, \ldots)\in F(\omega)$, we have $F(\omega)$ is a nonempty set.
	
	The quadratic form $\langle x, T x\rangle = \sum\limits_{n=1}^{\infty}\frac{x_n^2}{n^n}$ is not a Legendre form and  ${\rm Sol}(\omega)=\emptyset$  (see, \cite[Example 3.3]{DongTam15b}).
	
	Let $\omega^{\varepsilon}= (T^{\varepsilon}, c_1, \alpha_1)$, where  $T^{\varepsilon} = T+\varepsilon I$, $\varepsilon >0$ and $I$ is the identity operator on $\ell^2$. Since $\Vert T^{\varepsilon}- T\Vert = \varepsilon$, we have 
		$$\Vert \omega^{\varepsilon}-\omega\Vert = \max \{\Vert T^{\varepsilon}- T\Vert, \Vert c_1-c_1\Vert, \Vert \alpha_1 - \alpha_1\Vert \} = \Vert T^{\varepsilon}- T\Vert  \to 0 \mbox{ as } \varepsilon \to 0.$$
	We  have also 
		$$\langle x, T^{\varepsilon} x\rangle = \langle x, T x\rangle +\varepsilon \langle x, I x\rangle = \langle x, T x\rangle + \varepsilon\Vert x\Vert^2\geqslant \varepsilon\Vert x\Vert^2,$$
	because $\langle x, T x\rangle\geqslant 0$. Hence  $\langle x, T^{\varepsilon} x\rangle$ is a Legendre form.
	
	Consider the problem 
	\begin{align}\label{eq:not_usc1}
	\begin{cases}
	&\min  f(x, \omega^{\varepsilon})= \frac{1}{2}\langle x, T^{\varepsilon} x\rangle \\
	&\mbox{ s. t. } x=(x_1, x_2,\ldots)\in \ell^2, \,  g_1(x,\omega^{\varepsilon}) =  \langle c_1, x\rangle + \alpha_1\leqslant  0.
	\end{cases}
	\end{align}	
	Let ${\rm Sol}(\omega^{\varepsilon})$ denote the solution of \eqref{eq:not_usc1}.
	Since $\langle x, T^{\varepsilon} x\rangle$ is nonnegative and $\langle x, T^{\varepsilon} x\rangle = 0$ if and only  if $x= 0$, we have  ${\rm Sol} (QPR_{\omega^{\varepsilon}}) =\{0\}$. It follows from Lemma \ref{theo:sol_nonemp} that   ${\rm Sol}(\omega^{\varepsilon})$ is a nonempty set for every $\varepsilon$.
	
	We have shown that there exists a sequence $\{\omega^{\varepsilon}\}$ converging to $\omega$ such that ${\rm Sol}(\omega^{\varepsilon}) \neq \emptyset$. Taking $V= \emptyset$  we get ${\rm Sol}(\omega) \subset V$ and  ${\rm Sol}(\omega^{\varepsilon}) \not\subset V$.
	Hence ${\rm Sol}(\cdot )$ is not usc at $\omega = (T, c_1, \alpha_1)$. 
\end{example}

\begin{lemma}\label{le:per_Leg}
	Suppose that  $\langle x, T x\rangle$ is a Legendre form on ${\cal H}$. 
	Then, there exists an open neighborhood $\cal U$ of $T$ in space of continuous linear operators  ${\cal L(H)}$  such that for every 
	$T^{\prime}\in{\cal U}$, $\langle x, T^{\prime}x\rangle$ is also a Legendre form. 
\end{lemma}
\begin{proof}
	Since $\langle x, T x\rangle$ is a Legendre form, there exist  an elliptic form $\langle x, T_1x\rangle$   and a quadratic form of finite rank $\langle x, T_2 x\rangle$  such that
	$\langle x, T x\rangle = \langle x, T_1x\rangle + \langle x, T_2 x\rangle$ (see \cite[Proposition 3.79]{BS00}). 
	
	 Let $\alpha$ be a  positive number  such that 
	$\langle x, T_1x\rangle \geqslant \alpha \Vert x\Vert^2,\, \forall x\in {\cal H}.$
	Choose $\varepsilon > 0$ so that $\varepsilon < \alpha$.
	Define ${\cal U }:= \{T^{\prime}\in {\cal L(H)}\mid \Vert T-T^{\prime}\Vert < \varepsilon\}$.
	
	Let $T^{\prime}\in {\cal U}$. By
	the Cauchy-Schwarz inequality (see \cite[p. 29]{Bauschke2011}), we obtain
	\begin{align}\label{eq:per_lege1}
	\langle x, T x\rangle -\langle x, T^{\prime}x\rangle = \langle x, (T-T^{\prime})x\rangle \leqslant \Vert T-T^{\prime}\Vert\Vert x\Vert^2\leqslant \varepsilon\Vert x\Vert^2. 
	\end{align}
	Substituting $\langle x, T x\rangle = \langle x, T_1x\rangle +\langle x, T_2 x\rangle$ into  \eqref{eq:per_lege1} we obtain
	\begin{align}\label{eq:per_lege2}
	\langle x, T^{\prime}x\rangle - \langle x, T_2 x\rangle \geqslant \langle x, T_1x\rangle -\varepsilon\Vert x\Vert^2.
	\end{align}
	Combining   \eqref{eq:per_lege2}  with $\langle x, T_1x\rangle \geqslant \alpha \Vert x\Vert^2$ yields
		$$\langle x, T^{\prime}x\rangle - \langle x, T_2 x\rangle \geqslant (\alpha - \varepsilon) \Vert x\Vert^2.$$
	From this  it follows that
	$\widetilde{Q}(x) := \langle x, T^{\prime}x\rangle - \langle x, T_2 x\rangle$ is elliptic.
	
	We have shown that the quadratic form $\langle x, T^{\prime}x\rangle=\widetilde{Q}(x)+ \langle x, T_2 x\rangle$ is the sum of an elliptic quadratic form and
	a quadratic form of finite rank. By \cite[Proposition 3.79]{BS00},  $\langle x, T^{\prime}x\rangle$ is a Legendre form. The proof is complete. \qed
\end{proof}

The next theorem gives the necessary and sufficient condition for the lower
semicontinuity of  ${\rm Sol} (\cdot)$.
\begin{theorem}\label{theo:lsc}
	Consider the problem \eqref{eq:QP_w} where $\langle x, T x\rangle$ is a Legendre form. Then, ${\rm Sol}(\cdot)$ is lsc   at $\omega$  if and only if the following three conditions are satisfied:
	\begin{itemize}
		\item [{\rm (i)}] ${\rm Sol}\eqref{eq:P_0} =\{0\}$; 
		\item [{\rm (ii)}] the system $g_i(x, \omega)\leqslant 0, i=1, \dots, m,$ is regular;
		\item [{\rm (iii)}] the set ${\rm Sol} (\omega)$ is a singleton.
	\end{itemize}
\end{theorem}
\begin{proof}
	{\it Necessity.} 
	To prove ${\rm (i)}$ we assume the contrary that there exists $\overline{v}\in {\cal H}$, $\overline{v}\neq 0$  such that
	\begin{align}
	\label{eq:nec-lsc3}T_i \overline{v} = 0,\;  \langle c_i, \overline{v}\rangle \leqslant 0,\; i=1,\ldots, m, \;\;
	\langle \overline{v},  T\overline{v}\rangle\leqslant 0.
	\end{align}
	
	Since $F(\omega)\neq \emptyset$,  it follows from \eqref{eq:nec-lsc3} that $F(\omega)$ is unbounded.
	For $\varepsilon > 0$, put $T^{\varepsilon} = T - \varepsilon I$,  where $I$ is the identity operator on $\mathcal{H}$. 
	It follows from Lemma  \ref{le:per_Leg}  that  $\langle x, T^{\varepsilon} x\rangle$ is a Legendre form  for $\varepsilon >0$  small enough.
	We can check that 
	$\langle \overline{v}, T^{\varepsilon} \overline{v}\rangle < 0$. Let $\omega^k =(T^{\varepsilon},  c, T_1,\ldots, T_m, c_1,\ldots, c_m,\alpha_1,\ldots, \alpha_m)$. Then, for any $x\in F(\omega^k)$,
		$$f(x+ t\overline{v}, \omega^k) = \frac{1}{2}\langle x+t\overline{v}, T^{\varepsilon}(x+t\overline{v})\rangle + \langle c, x+t\overline{v}\rangle \to -\infty \mbox{ as } t\to\infty.$$
	Thus ${\rm Sol}(\omega^k) = \emptyset$.
	This contradicts our assumption that ${\rm Sol}( \cdot )$ is lower semicontinuous at $\omega$.
	
	If ${\rm (ii)}$ does not hold, then
	one can find an  $\omega^k$ arbitrarily close to $\omega$ such that
		$$F(\omega^k) = \{x\in {\cal H}\mid g_i(x, \omega^k)\leqslant 0\}\;\; \mbox{is empty}.$$
	Hence, we can find a parameter $\omega^k=( T, c, T^k_1, \ldots, T^k_m, c^k_1, \ldots, c^k_m, \alpha_1^k, \ldots, \alpha_m^k)$  that 
	close to $\omega=( T, c, T_1, \ldots, T_m,  c_1,\ldots, c_m, \alpha_1, \ldots, \alpha_m)$ such that
	$F(\omega^k)$  is empty. Since ${\rm Sol}(\omega^k) = \emptyset$, ${\rm Sol}(\cdot)$ cannot be lower
	semicontinuous at $\omega$.
	
	Suppose that ${\rm Sol}(\cdot)$ is lower semicontinuous at $\omega$,
	but ${\rm Sol}(\omega)$ is not a singleton.
	Since ${\rm Sol} (\omega)\neq\emptyset$, there exist $\bar x,\bar y\in {\rm Sol}(\omega)$
	such that $\bar x\not=\bar y$. By the Hahn-Banach Theorem (see \cite[Theorem 1.7]{Brezis}) there exists $\bar c\in{\cal H}$ such that
		\begin{align}\label{eq:nec-lsc1}
	\Vert \bar c\Vert =1, \quad 
	\langle\bar c, \bar x\rangle >\langle \bar c, \bar y\rangle. 
	\end{align}
	Clearly, there exists an open neighborhood $U$ of $\bar x$ such that
	\begin{align}\label{eq:nec-lsc2}
	\langle \bar c, x\rangle> \langle \bar c, \bar y\rangle \quad \text{for all}\quad x\in U. 
	\end{align}
	Given any $\delta>0$, we fix a number $\varepsilon\in (0,\delta)$
	and put $c^\varepsilon=c+\varepsilon \bar c$. By \eqref{eq:nec-lsc1}, $\Vert
	c^\varepsilon-c\Vert=\varepsilon <\delta.$ Our next goal is to show
	that
		$${\rm Sol}(\omega^{\varepsilon})={\rm Sol} (T, c^\varepsilon, T_1,\ldots, T_m,  c_1,\ldots, c_m, \alpha_1,\ldots, \alpha_m)\cap U=\emptyset.$$
	For any
	$x\in F(\omega) \cap U$, since $\bar x,\bar y\in {\rm Sol}(\omega)$, by \eqref{eq:nec-lsc2} we have
	\begin{eqnarray*}
		\frac{1}{2}\langle x, Tx\rangle + \langle c^\varepsilon, x \rangle
		&=&\frac{1}{2}\langle x, Tx\rangle+\langle(c+\varepsilon \bar c), x\rangle =(\frac{1}{2}\langle x, Tx\rangle +\langle c, x\rangle)+\varepsilon \langle \bar c, x\rangle\\
		&\geqslant & \frac {1}{2}\langle\bar x, T\bar x\rangle+\langle c,\bar x\rangle+\varepsilon \langle\bar c, x\rangle>\frac {1}{2}\langle\bar y, T\bar y\rangle+\langle c,\bar y\rangle+\varepsilon \langle\bar c, \bar y\rangle\\
		&=& \langle \bar y, T\bar y\rangle +\langle c^{\varepsilon}, \bar y\rangle.
		\end {eqnarray*}
		It follows that $x\not\in {\rm Sol}(\omega^{\varepsilon})$. Thus,
		for the chosen neighborhood $U$ of $\bar x\in {\rm Sol}(\omega^{\varepsilon})$
		and for every $\delta>0$, there exists $c^\varepsilon\in { \cal H}$
		satisfying $\Vert c^\varepsilon-c\Vert<\delta$ and ${\rm
			Sol}(\omega^{\varepsilon})\cap U=\emptyset$. This contradicts the
		lower semicontinuity  of ${\rm Sol} (\cdot)$. Hence
		${\rm Sol}(\omega)$ is a singleton.

		{\it Sufficiency.} Let $U$ be an open set in $\cal H$ containing the unique solution  $\overline{x}\in {\rm Sol}(\omega)$.
		By   ${\rm (ii)}$, there exists $\delta_1>0$ such that $F(\omega^{\prime})\neq \emptyset$, for every  $\omega^{\prime}\in\Omega$ satisfying 
		$\Vert \omega^{\prime}- \omega\Vert < \delta_1.$
		
		By ${\rm (i)}$, and Lemmas \ref{le:per_Leg}, \ref{pro:open}, there exists $\delta_2$
		such that    for every 
				$$(T^{\prime}, T^{\prime}_1,\ldots,  T^{\prime}_1, c^{\prime}_1,\ldots, c^{\prime}_m)\in {\cal L(H)}^{m+1}\times {\cal H}^m$$
			satisfying 
				$$\max\{\Vert T^{\prime}-T \Vert, \Vert T_1^{\prime}-T_1 \Vert,\ldots, 
			\Vert T_m^{\prime}-T_m \Vert , 
			\Vert c_1^{\prime}- c_1\Vert, \ldots,\Vert c_m^{\prime}- c_m\Vert \}< \delta_2,$$
			$\langle x,  T^{\prime}x\rangle$ is a Legendre form and  ${\rm Sol}( QPR_{\omega^{\prime}}) = \{0\}$. 
		Let $\delta := \min \{\delta_1, \delta_2\}$.
		By Lemma \ref{theo:sol_nonemp},  for every
		$\omega^{\prime}$  satisfying
		\begin{align}\label{eq:prooflsc}
		\Vert \omega^{\prime}-\omega\Vert\leqslant \delta,
		\end{align}
		we have ${\rm Sol}(\omega^{\prime}) \neq \emptyset$.
		From ${\rm (i), \; (ii)}$ and   Theorem \ref{theo:usc} it follows that ${\rm Sol}(\cdot)$ is upper semicontinuous at $\omega$. Hence, for $\delta >0$  small enough,  ${\rm Sol} (\omega^{\prime})\subset U$  for every $\omega^{\prime}$ satisfying \eqref{eq:prooflsc}.
		
		For such an $\delta> 0$, from what has been said it follows that ${\rm Sol} (\omega^{\prime})\cap \Omega \neq \emptyset$ for every $\omega^{\prime}$ satisfying  \eqref{eq:prooflsc}. 
		This shows that ${\rm Sol}(\cdot)$ is lower semicontinuous at $\omega$. The proof is complete. \qed
	\end{proof}
	The following example shows that  the assumption on the Legendre property of the quadratic form cannot be dropped from the assumption of Theorem \ref{theo:lsc}.
	\begin{example}\rm 
		Consider the problem \eqref {ex:unboundedSet}  in Example \ref{exam:unbounded}. Let   ${\rm Sol}(\omega)$ denote  the solution set of  \eqref {ex:unboundedSet}. It follows from Example \ref{exam:unbounded} that ${\rm Sol}(\omega)$ is unbounded.  
		
		Let $\omega^{\varepsilon}= (T^{\varepsilon}, T_1, \alpha)$ where  $T^{\varepsilon} = 0+\varepsilon E$, $\varepsilon >0$ and $E$ is the identity operator on $L_2[0,1]$. Since $\Vert T^{\varepsilon}- T\Vert = \varepsilon$, we have 
			$$\Vert \omega^{\varepsilon}-\omega\Vert = \max \{\Vert T^{\varepsilon}- T\Vert, \Vert T_1-T_1\Vert, \vert \alpha_1 - \alpha_1\vert \} = \Vert T^{\varepsilon}- T\Vert  \to 0 \mbox{ as } \varepsilon \to 0.$$
			Consider the  problem 
		\begin{align}\label{ex:not_usc}
		\begin{cases}
		& \min  f(x,\omega^{\varepsilon})= \frac{\varepsilon}{2}\Vert x\Vert^2\\
		&\mbox{ s. t. }   x\in L_2[0,1],  g_1(x, \omega^{\varepsilon}) = \frac{1}{2} \langle x, T_1 x\rangle +\alpha_1\leqslant  0.
		\end{cases}
		\end{align}
		Let   ${\rm Sol}(\omega^{\varepsilon})$ denote  the solution set of \eqref{ex:not_usc}.
		It is clear that ${\rm Sol}(\omega^{\varepsilon}) = \{0\}$.  Let $V$ be an open set in $\cal H$ such that  $V\cap {\rm Sol}(\omega) \neq \emptyset$ and $0 \notin V$. Then,  $V\cap {\rm Sol}(\omega^{\varepsilon}) = \emptyset$. Hence ${\rm Sol}(\cdot )$ is not lsc at $\omega=(T, T_1,\alpha_1)$. 
	\end{example}  
	\begin{corollary}
		Consider the problem \eqref{eq:QP_w} where  $\langle x, T x\rangle$ is a nonpositive Legendre form, then the multifunction 
		$\rm Sol( . )$ is lower  secontinuous at $\omega$ if and only if  the following
		conditions are satisfied:
		\begin{itemize}
			\item [{\rm (i)}] $F(\omega)$ is a compact set;
			\item [{\rm (ii)}]  the system $g_i(x, \omega)\leqslant 0$, $i=1,\ldots, m$, is regular;
			\item [{\rm (iii)}] the set ${\rm Sol} (\omega)$ is a singleton.
		\end{itemize}
	\end{corollary}
	\begin{proof}
		Suppose that $\rm Sol( \cdot )$ is lower  semicontinuous at $\omega$.
		Since $\rm Sol( \cdot )$ is lower semicontinuous at $\omega$, by Theorem
		\ref{theo:lsc}, conditions ${\rm (ii)}$ and  ${\rm (iii)}$ are satisfied and
		${\rm Sol}\eqref{eq:P_0} =\{0\}.$
		We now claim that $0^+F(\omega)=\{0\}$. Indeed, by $\langle x, T x\rangle$ is  nonpositive, we have $\langle v, T v\rangle \leqslant 0$ for every $v\in 0^+F(\omega)$. If there exists no $\bar v\in 0^+F(\omega)$ with property that $\langle\bar v, T \bar v\rangle < 0$ then ${\rm Sol}\eqref{eq:P_0} =\{0\}= 0^+F(\omega)$. If $\langle\bar v, T \bar v\rangle < 0$ for some $\bar v\in 0^+F(\omega)$ then it is obvious that ${\rm Sol}\eqref{eq:P_0} =\emptyset$, which is impossible. Our claim is proved. 	Since $\langle x, T x\rangle$ is a nonpositive Legendre form, $\cal H$ is of finite dimension (see, \cite[Theorem 11.2]{Hestenes1}). Hence, since $F(\omega)$ is  a nonempty closed set and $0^+F(\omega)=\{0\}$, $F(\omega)$ is a compact set. 
		
		Conversely, suppose that ${\rm (i)}, {\rm (ii)}$  and  ${\rm (iii)}$ are satisfied. As $F(\omega)\neq \emptyset$ by  assumption ${\rm (i)}$ implies that $0^+F(\omega)= \{0\}$. Therefore, ${\rm Sol}\eqref{eq:P_0} =\{0\}$. Since the conditions ${\rm (i), (ii)}$ and ${\rm (iii)}$ in Theorem \ref{theo:lsc} are satisfied, we conclude that $\rm Sol( \cdot )$ is lsc at  $\omega$. The proof is complete \qed
	\end{proof}

	\section{Continuity of the Optimal Value Function}
	
	We now present a set of two conditions which is necessary and sufficient for the continuity of $\varphi$ at a point $\omega$ where $\varphi$ has a finite value.
	\begin{theorem}\label{theo:covf}
		Consider the problem $(QP_{\omega})$ where $\langle x, T x\rangle$ is a Legendre form.
		Let $\omega\in \Omega$. Suppose that $\varphi (\omega)\neq\pm \infty$. Then, the optimal value function $\varphi (\cdot)$ is continuous at $\omega$ if and only if the following two conditions are satisfied:
		\begin{itemize}
			\item [{\rm (i)}] ${\rm Sol}\eqref{eq:P_0} =\{0\}$;
			\item [{\rm (ii)}] The system $g_i(x, \omega)\leqslant 0, i=1,\ldots, m$, is regular.
		\end{itemize} 
	\end{theorem}
	\begin{proof}
		Suppose that $\varphi (\cdot)$ is continuous at $\omega$ and $\varphi (\omega)\neq\pm \infty$. 
		
		If the system $g_i(x, \omega)\leqslant 0, i=1,\ldots, m$,  is irregular, then  by Remark \ref{remark:lscF} there exists  a sequence 
		\begin{center}
			$\omega^k:= \{(T^k_1,\ldots, T^k_m, c^k_1,\ldots, c^k_m, \alpha^k_1,\ldots, \alpha^k_m)\}\in {\cal L(H)}^m\times {\cal H}^m\times \mathbb R^n $
		\end{center}
		converging to 
		$ \omega:= (T_1,\ldots, T_m, c_1,\ldots, c_m, \alpha_1,\ldots, \alpha_m)$
		such that, for every $k$, 
		\begin{center}
			$g_i(x, \omega^k):=\frac{1}{2}\langle x, T^k_i x\rangle + \langle c^k_i, x\rangle +\alpha^k_i\leqslant 0,\; i=1,\ldots, m$
		\end{center}
		has no solution. Since $F(\omega^k)$ is empty for every $k$,  $\varphi (\omega^k) = +\infty.$ 
		On the other hand, since $\varphi (\cdot)$ is continuous at 
		$\omega$,
		$\varphi(\omega)= \underset{k\to\infty}{\lim}\varphi(\omega^k) = +\infty$. We have arrived at a contradiction. This shows that $g_i(x, \omega)\leqslant 0$ is regular.
		
		We now suppose that ${\rm Sol}\eqref{eq:P_0} \neq \{0\}$. Then, there is a nonzero $\overline{v}\in {\cal H}$ such that
					$$ \langle \bar v, T \bar v\rangle \leqslant 0, \;\;  
			T_i\bar v = 0,\; \langle c_i, \bar v\rangle \leqslant 0,\;\; i=1,\ldots, m. $$
			Define $T^k = T-\frac{1}{k} E$, where $E$ is the identity operator on $\cal H$.  It follows from Lemma \ref{le:per_Leg} that  $\langle x, T^k x\rangle$ is a Legendre form for  $k$ large enough. 
		It is easy to check that  
		$\langle \bar v, T^k \bar v \rangle < 0$.
		Consider the sequence $\{\omega^k\}$,
				$$\omega^k = (T^k, c, T_1,\ldots,T_m,  c_1,\ldots,c_m,\alpha_1,\ldots,\alpha_m).$$
				From the assumption $\varphi (\omega)\neq \pm\infty$, it follows that $F(\omega)\neq \emptyset$. 
		Hence, for any $x\in F(\omega)$ and for any $t> 0$, we have $x+t\bar v\in F(\omega)$ and 
				$$f(x+t\bar v, \omega^k) = \frac{1}{2}\langle x+t\bar v, T^k (x+t\bar v) \rangle+\langle c, x+t\bar v\rangle\to -\infty \mbox{ as } t\to\infty.$$
			This implies that, for every $k$ large enough, ${\rm Sol}(\omega^k)= \emptyset$ and $\varphi (\omega^k) = -\infty$. We arrived at a contradiction, because $\varphi (\cdot)$ is continuous at $\omega$, $\omega^k$ converges to $\omega$ and $\varphi (\omega)\neq \pm\infty$. 	
		
		From now on we assume that ${\rm (i)}$ and  ${\rm (ii)}$ are satisfied and $\omega^k$ is an arbitrarily sequence in 
		${\cal L(H)}^{m+1}\times {\cal H}^{m+1}\times \mathbb R^m$ converging to $\omega$. 
		By ${\rm (ii)}$ and Lemma \ref{pro-lscF},  there   exists a positive integer $k_0$ such that $F( \omega^k)\neq \emptyset$ for every $k\geqslant k_0$. By ${\rm (i)}$, Lemma \ref{pro:open} and  Lemma \ref{le:per_Leg}, there exists  positive integer $k_1\geqslant k_0$ such that for all $k\geqslant k_1$,  $\langle x, T^k x\rangle$ is a Legendre form and 
		${\rm Sol}(QPR_{\omega^k}) =\{0\}$. It follows from Lemma  \ref{theo:sol_nonemp} that
		${\rm Sol}(\omega^k)\neq\emptyset$ with all $k$ large enough.
		Therefore, for every $k\geqslant k_1,\; \varphi (\omega^k)$ is finite. This means that, for every $k\geqslant k_1$, there exists $x^k\in {\cal H}$ satisfying 
		\begin{center}
			$\varphi(\omega^k) = \frac{1}{2}\langle x^k, T^k x^k\rangle + \langle c^k, x^k\rangle,	\; 	 g_i(x, \omega^k)\leqslant 0.$
		\end{center}
		By ${\rm (i)}$ and Lemma \ref{theo:sol_nonemp}, it follows that ${\rm Sol}(\omega)\neq \emptyset$. 
		Taking any $x_0\in {\rm Sol}(\omega)$, we have 
		\begin{align*}
		&\varphi(\omega) =\frac{1}{2} \langle x_0, T x_0\rangle +\langle c, x_0\rangle,\\ 
		& g_i(x_0, \omega) =\frac{1}{2} \langle x_0, T_i x_0 \rangle + \langle c_i, x_0\rangle + \alpha_i\leqslant 0,\;\; i=1,\ldots, m.
		\end{align*}  
		By  ${\rm (ii)}$ and Lemma \ref{pro-lscF}, there exists a sequence $\{y_k\}\subset {\cal H}$ converging to $x_0$ and 
		\begin{align}\label{eq:theo_covf3}
	\hspace*{-0.23truecm}	g_i(y_k,\omega) = \frac{1}{2}\langle y_k, T_i y_k \rangle + \langle c_i, y_k\rangle + \alpha_i\leqslant 0,\, i=1,\ldots,m \mbox{ for every } k\geqslant k_1.
		\end{align} 
		From \eqref{eq:theo_covf3} it follows that $y_k\in F(\omega^k)$ for  $k\geqslant k_1$. Then
		\begin{align}\label{eq:theo_covf4a}
		\varphi (\omega^k)\leqslant \frac{1}{2}\langle y_k, T^k y_k\rangle + \langle c^k, y_k\rangle.
		\end{align}
		It follows from \eqref{eq:theo_covf4a}  that
		\begin{alignat}{2} \label{eq:theo_covf4}
		&\underset{k\to\infty}{\limsup}\,\varphi (\omega^k)\; & \leqslant \; &  \underset{k\to\infty}{\limsup}\Big(\frac{1}{2}\langle y_k, T^k y_k\rangle + \langle c^k, y_k\rangle\Big) 
		=\underset{k\to\infty}{\lim}\Big(\frac{1}{2}\langle y_k, T^k y_k\rangle + \langle c^k, y_k\rangle\Big) \notag{}\\
		& &=& \frac{1}{2}\langle x_0, T x_0\rangle + \langle c, x_0\rangle= \varphi (\omega).
		\end{alignat}
		We now claim that the sequence $\{x^k\}$ is bounded. Indeed, if it is unbounded then, by taking a subsequence if necessary, we can assume that   $\Vert x^k\Vert\neq 0$ for every $k$ and $\Vert x^k\Vert\to\infty$ as $k\to\infty$. Then, the sequence $\{v^k\}=\{\Vert x^k\Vert^{-1}x^k\}$ is bounded. Without loss of generality, we can assume
		that $v^k$ itself weakly converges to some $v$.
		It is easy to check that $v\in 0^+F(\omega)$. 
		By dividing both sides of the inequality 
				$$\frac{1}{2}\langle x^k, T^k x^k\rangle + \langle c^k, x^k\rangle 
			\leqslant \frac{1}{2}\langle y_k, T^k y_k\rangle+ \langle c^k, y_k\rangle$$
			by $\Vert x^k\Vert^{2}$, letting $k\to\infty$ and by Lemma \ref{le:converge}, we get
				$$\langle v, T v\rangle\leqslant \underset{k\to\infty}{\lim\inf}\langle v^k, T^k v^k\rangle\leqslant 
			\underset{k\to\infty}{\lim\sup}\langle v^k, T^k v^k \rangle\leqslant 0.$$
			By an argument analogous to that used in the proof of Lemma \ref{pro:open}, 	 we have $v\in 0^+F({\omega})\setminus \{0\}$ and 	$\langle v, Tv\rangle \leqslant 0$,   contrary to ${\rm (i)}$. 
		We have thus shown that the sequence $\{x^k\}$ is bounded; hence it has a weakly convergent subsequence. Without loss of generality, we may assume that $x^k$ weakly converges to $ \bar x$. 	By Lemma \ref{le:converge}, we get  $\bar x\in F(\omega)$ and 
		\begin{equation}\label{eq:theo_covf5}
		\underset{k\to\infty}{\liminf}\,\varphi (\omega^k) =\underset{k\to\infty}{\liminf}\big(\langle x^k, T^k x^k\rangle+\langle c^k, x^k\rangle\big)
		\geqslant \langle \bar x, T \bar x\rangle +\langle c, \bar x\rangle
		\geqslant \varphi (\omega).
		\end{equation}
		Combining \eqref{eq:theo_covf4} with \eqref{eq:theo_covf5} gives
		$\underset{k\to\infty}{\lim}\varphi (\omega^k) = \varphi (\omega).$
		This shows that $\varphi (\cdot)$ is continuous at $\omega$. The proof is complete. \qed
	\end{proof}
	\begin{remark}
		Note that  if the assumption on  Legendre property of the quadratic form is omitted, then  the problem \eqref{eq:QP_w} may have no solution (see, \cite[Example 3.3]{DongTam15b}).  Hence the conclusion of  Theorem \ref{theo:covf} fails if the assumption on  Legendre property of the quadratic form is omitted.
	\end{remark}

	\begin{theorem}\label{theo:stability}
		Consider the problem \eqref{eq:QP_w} where $\langle x, T x\rangle$ is a Legendre form.
		The following four statements are equivalent:
		\begin{itemize}
			\item [{\rm ($\rm \beta_1$)}] the solution map ${\rm Sol} (\cdot)$  is lower semicontinuous at $\omega$;
			\item [{\rm ($\rm \beta_2$)}]   the solution map ${\rm Sol} (\cdot)$ is continuous at $\omega$;
			\item [{\rm ($\rm \beta_3$)}] ${\rm Sol} (\omega)$ is a singleton and the optimal value function $\varphi ( \cdot )$ is Lipschitz continuous around
			$\omega$;
			\item [{\rm ($\rm \beta_4$)}] ${\rm Sol} (\omega)$ is a singleton and the optimal value function $\varphi ( \cdot )$ is continuous 	at $\omega$.
		\end{itemize}	
	\end{theorem}
	
	\begin{proof}
		
	The equivalence between  ${\rm (\beta_1)}$ and  ${\rm (\beta_2)}$ follows immediately from  the Theorem \ref{theo:usc} and \ref{theo:lsc}.

		 We next prove that $\rm (\beta_2)$  implies $\rm (\beta_3)$. Indeed, suppose that the solution map ${\rm Sol} (\cdot)$ is continuous at $\omega$. Then, solution map ${\rm Sol} (\cdot)$ is lower semicontinuous at $ \omega$.   It follows from Theorem \ref{theo:lsc} that  ${\rm Sol} (\omega)$ is a singleton, ${\rm Sol}\eqref{eq:P_0} =\{0\}$ and the system $g_i(x, \omega)\leqslant 0$ is regular.
		It remains to show that  $\varphi ( \cdot )$  is locally Lipschitz at $\omega$.
		Since $f(\cdot \, , \cdot)$ is continuously differentiable at $(x, \omega)$,  there exists $\delta > 0$ so that $f$ is Lipschitz continuous with Lipschitz modulus $k_f > 0$ on the set $\mathbb U^{\delta}_{{\cal H}\times \Omega}( x,  \omega)$. 
By the regularity of the inequalities system  $g_i(x, \omega)\leqslant 0,$  the feasible set mapping $F\, (\cdot): \Omega \rightrightarrows {\cal H}$ is defined by 
		$F\, (\omega^{\prime}) = \{x\in {\cal H}\mid g_i(x, \omega^{\prime})\leqslant 0,\; i=1,\ldots, m\}$  has the Aubin
		property at $ \omega$ for some $  x\in F(\omega)$  (see, for instance, \cite[Corollary 2.2]{Dempe}), that
		is, there exist $\varepsilon, \gamma, k_F\geqslant 0$		 such that
		\begin{align}\label{eq:lips1}
		F(\omega^1)\cap U^{\gamma}( x)\subset F(\omega^2)+k_F\Vert \omega^1-\omega^2\Vert {\mathbb B}_{\Omega}(0, 1)
		\end{align}
		for every $\omega^1, \omega^2\in \mathbb U^{\varepsilon}_{\Omega}( \omega)$.		
		Without loss of generality, we may assume that  $\max\{\varepsilon, 2\varepsilon k_F+\gamma\} < \delta$.
		
		Let $\omega^1, \omega^2\in \mathbb U^{\varepsilon}_{\Omega}( \omega)$ be chosen arbitrarily.    By Lemma \ref{pro-lscF}, the feasible set mapping $F (\cdot): \Omega \rightrightarrows {\cal H}$ is lower semicontinuous at $\omega$. This implies that $F(\omega^1) \neq\emptyset$ for $\varepsilon > 0$ small enough. On the other
		hand, by ${\rm Sol}\eqref{eq:P_0} =\{0\}$, it follows that $\omega \in K$. Since $K$ is open, there exists $\varepsilon > 0$  small
		enough such that $\omega^1\in K $. It follows from Lemma \ref{theo:sol_nonemp} that ${\rm Sol}(\omega^1)\neq\emptyset$, for $\varepsilon > 0$  small
		enough. Hence there exists $x^1\in {\rm Sol}(\omega^1)$.
		
		Since $\rm Sol( \cdot )$ is lsc at $\omega$, we can assume that $\varepsilon > 0$ small enough to guarantee that $x^1\in U^{\gamma}( x)$. This implies that $x^1\in {\rm Sol(\omega^1)}\cap  U^{\gamma}( x)\subset F(\omega^1)\cap U^{\gamma}( x)$.	
		
		Due to \eqref{eq:lips1}, there exists $x^2\in F(\omega^2)$ such that
		$\Vert x^2-x^1\Vert \leqslant k_F\Vert \omega^2-\omega^1 \Vert$.
		From the choice of $\varepsilon$ and $\delta$ we derive:
		\begin{alignat*}{2}
		&\Vert ( x^1, \omega^1) - (  x, \omega)\Vert& =&\max\{\Vert \omega^1-\omega\Vert, \Vert x^1- x \Vert \}
		<\max\{\varepsilon, \gamma\} < \delta,\\
		&\Vert ( x^2, \omega^2) - (  x, \omega)\Vert&=&\max\{\Vert \omega^2-\omega\Vert, \Vert x^2- x \Vert\}\\ 
		& &\leqslant& \max\{\Vert \omega^2-\omega\Vert, \Vert x^2- x^1\Vert +\Vert x^1-  x\Vert \}\\ 
		& &\leqslant& \max\{\Vert \omega^2-\omega\Vert, k_{F}\Vert \omega^2- \omega^1\Vert +\Vert x^1-  x\Vert \}\\ 
		& &<&\max\{\varepsilon, 2\varepsilon  k_F  + \gamma\} < \delta.
		\end{alignat*}
		
		Thus $( x_1, \omega_1)$ and $( x_2, \omega_2)$ belong to $\mathbb U^{\delta}_{{\cal H}\times \Omega}( x, \omega)$.
		
		Since $f$ is Lipschitz continuous on $\mathbb U^{\delta}_{{\cal H}\times \Omega}( x,  \omega)$, we finally get
		\begin{align*}
		\varphi(\omega_2)-\varphi(\omega_1)&\leqslant f(\omega_1, x_2) - f(\omega_1, x_1)
		\leqslant k_f \max\{\Vert \omega_2-\omega_1\Vert; \Vert x_2-x_1\Vert\}\\ 
		&\leqslant k_f \max\{\Vert \omega_2-\omega_1\Vert; k_{F}\Vert \omega_2-\omega_1\Vert\}= \max\{k_f; k_f k_F\}\Vert \omega_2-\omega_1\Vert.
		\end{align*}
		Changing the roles of $x_1$ and $x_2$, we can obtain
		\begin{align*}
		\varphi(\omega_1)-\varphi(\omega_2) \leqslant \max\{k_f; k_f k_F\}\Vert \omega_2-\omega_1\Vert.
		\end{align*}
		Hence $\varphi$ is  Lipschitz continuous around $ \omega$ with modulus $\max\{k_f; k_f k_F\}$.
		
		Clearly,  $\rm (\beta_3)$  implies $\rm (\beta_4)$.
		
		Finally, the implication $\rm (\beta_4)$  implies $\rm (\beta_1)$ follows from Theorem \ref{theo:lsc} and \ref{theo:covf}. The proof is complete. \qed
	\end{proof}
	
	\begin{example}\rm
		We consider the problem \eqref{eq:QP_w} with 
		$\Omega={\cal L}(\ell^2)\times {\ell^2}\times{\cal L}(\ell^2)\times {\ell^2}\times\mathbb R$.
		
		Let $\omega= (T, c, T_1, c_1, \alpha_1)$, where $T: \ell^2\to \ell^2$ is defined by $Tx=(0, -x_2, x_3,\ldots)$, $c=(1,0,\ldots)$, $T_1: \ell^2\to \ell^2$ is defined by $T_1x= (x_1, x_2,\ldots)$, $c_1=(0, 0, \ldots)$, $\alpha_1= -\frac{1}{2}$. This problem can be rewritten as follows
		\begin{center}
			$\underset{x\in F(\omega)}{\min} \; f(x, \omega) =  \frac{1}{2}(-x_2^2 + x_3^2+\ldots x_n^2 + \ldots) + x_1,$ 
		\end{center}
		where $F(\omega) = \{(x_1,x_2,\ldots)\in \ell^2\mid g_1(x, \omega) = \frac{1}{2}(x_1^2+x_2^2+\ldots) -\frac{1}{2}\leqslant 0\}$.
		
		Since $\langle x, Tx\rangle= (x_1^2+x_2^2+x_3^2+\ldots+x_n^2+\ldots) -(x_1^2+ 2x_2^2)$, by Proposition 3.79 in \cite{BS00},  $\langle x, Tx\rangle$ is a Legendre form.
		
		For $x= (0, 0,\ldots)$, we have $g_1(0,\omega) =-\frac{1}{2} < 0$. Hence $F(\omega)\neq 0$ and $g_1(x,\omega)\leqslant 0$ is regular.
		
		We have $0^+F(\omega) = \{(v_1,\ldots)\in \ell^2\mid T_1v= (v_1, v_2, \ldots) = 0\}= \{0\}$. Therefore ${\rm Sol}\eqref{eq:P_0} =\{0\}$.
		
		It is easy to check that $f(x, \omega) \geqslant \frac{1}{2}(x_1+1)^2 +(\frac{1}{2}x_2^2+ x_3^2+\ldots)- 1\geqslant -1$ for all $x\in F(\omega)$ and ${\rm Sol}(\omega) = \{-1,0,0,\ldots\}$.
		
		Hence, by Theorem \ref{theo:stability}, the solution map ${\rm Sol}(\cdot)$ is continuous at $\omega$ and the optimal value function $\varphi (\cdot)$ is Lipschitz continuous around $\omega$.
	\end{example}
	
	\begin{corollary}
		Consider the problem \eqref{eq:QP_w} where $\langle x, T x\rangle$ is elliptic. Suppose that the system $g_i(x, \omega)\leqslant 0$, $i=1,\ldots, m$, is regular. Then, the following four statements are equivalent:
		\begin{itemize}
			\item [{\rm ($\beta_1$)}] the solution map ${\rm Sol} (\cdot)$  is lower semicontinuous at $\omega$;
			\item [{\rm ($\beta_2$)}]   the solution map ${\rm Sol} (\cdot)$ is continuous at $\omega$;
			\item [{\rm ($\beta_3$)}] the optimal value function $\varphi ( \cdot )$  is Lipschitz continuous around $\omega$;
			\item [{\rm ($\beta_4$)}] the optimal value function $\varphi ( \cdot )$ is continuous 	at $\omega$.
		\end{itemize}
	\end{corollary}
	
	The proof of this corollary is simple, so it is omitted.

\section{Conclusions}
By using the Legendre property of quadratic form,  we established  continuity properties  of the global solution map and the optimal value function for  quadratic programming problems under finitely many convex quadratic constraints in Hilbert spaces. 

In connection with Theorem \ref{theo:lsc},  the following question seems to be interesting:	Is there any verifiable  sufficient condition  for the set ${\rm Sol} (\omega)$ to be a singleton?

\section*{Acknowledgments}
This research is funded by Vietnam National Foundation for Science and Technology Development (NAFOSTED) under grant number 101.01-2014.39. 
The author would like to thank Prof. Nguyen Nang Tam for valuable remarks and suggestions. 
The author would like to express my sincere thanks to the anonymous referees and  editors for insightful comments and useful suggestions.


\end{document}